\documentclass[12pt]{article}
\usepackage[utf8]{inputenc}
\usepackage[T2A, T1]{fontenc}

\usepackage{epsfig}
\usepackage[a4paper, width=170mm,top=20mm,bottom=20mm]{geometry}
\usepackage{amsfonts}
\usepackage{amsmath}
\usepackage{amssymb}
\usepackage{enumerate}
\usepackage{multicol}
\usepackage{amsthm}

\theoremstyle{definition}
\newtheorem{theorem}{Theorem}
\newtheorem{proposition}{Proposition}
\newtheorem{lemma}{Lemma}
\newtheorem{corollary}{Corollary}

\title{\textbf{Solutions of time fractional anomalous diffusion equations with coefficients depending on both time and space variables}}
\author{Ganbileg Bat-Ochir, Khongorzul Dorjgotov, Uuganbayar Zunderiya}
\date{\today}
\begin{document}
\maketitle
\begin{abstract}
  We derive explicit solutions for time-fractional anomalous diffusion equations with diffusivity coefficients that depend on both space and time variables. These solutions are expressed in Fox-H and generalized Wright functions, which are commonly used in anomalous diffusion equations. Our study represents a significant advancement in our understanding of anomalous diffusion with potential applications in a wide range of fields.
\end{abstract}
\section{Introduction}

Anomalous diffusion has been the subject of extensive research in recent years, with numerous publications that address different aspects of this phenomenon. The growing number of applications that can be described by anomalous diffusivity has led to increased interest in this area of study. Until now, previous research has explored the dependence of diffusion coefficients on space, but no published results have provided explicit solutions when the diffusion coefficients depend on both space and time variables. This work builds on previous research \cite{old1974,wyss1986,schneider1989,metz94,main96,Luchko,pod1999,metzler2000,main03,kilb2005,kilbas2006,main07,main10,tar2019}, which provided exact solutions for anomalous diffusion equations with a diffusion coefficient function that depended only on the space variable. We extend our previous work to include the time variable in the diffusion coefficient function. This development allows for the modeling of processes that were previously inaccessible.

In \cite{hents}, the authors proposed that the following diffusion equation
\begin{equation*}
\frac{\partial }{\partial t}P(x,t)=\frac{\partial }{\partial x}\left(K(x,t)\frac{\partial}{\partial x}P(x,t)\right),  
\end{equation*}
where the diffusivity coefficient  is taken as $K(x,t)=cx^at^b,$ $a,b,c\in \mathbb{R}$,  to describe diffusion processes in turbulent media. 

In \cite{metz94}, authors proposed time fractional version of the equation with the coefficient $K$ depends only on $x,$ i.e. $K(x)=cx^a,$ to describe anomalous diffusion with an asymptotic behaviour which is regarded as a general property of diffusion in fractal structure. Since it has been shown that the diffusivity coefficient should not depend only on spatial variable in \cite{hents}, we consider the following generalized equation
\begin{equation}\label{appl}
\frac{\partial^\alpha}{\partial t^\alpha}u(x,t)=
  t^m\left(Ax^{d}\frac{\partial^2}{\partial x^2}u(x,t)+Bx^{d-1}\frac{\partial}{\partial x}u(x,t)+Cx^{d-2}u(x,t)\right),\quad x>0,~ t>0.  
\end{equation}
Here $m$ is a nonnegative integer, $d$ is a real number, $A$ is a positive real number, $B$ and $C$ are real numbers.

Equation (\ref{appl}) can be considered as a time fractional generalization of equation given in \cite{hents} as well as a generalization of the equation studied in \cite{metz94} considering the diffusivity coefficient depends both on time and space variables.

The solutions that will be presented in this work will be expressed in special functions, such as generalized Wright function and Fox-H functions. So we refer the readers who are not familiar with these special functions should consult with the literature \cite{Hnom1,Hnom2}.

\section{Main result}

We recall the following special functions and their properties in preparation for introducing exact invariant solutions of Eq.~(\ref{appl}).
The Fox H-function (e.g. in \cite{kilb2005,Hnom1,Hnom2}) is defined by means of the Mellin-Barnes type contour integral 
\begin{equation}
\label{int1}
H_{p,q}^{m,l}\left[z\biggr\vert\begin{array}{c}
(a_i, \alpha_i)_{1,p}\\
(b_j, \beta_j)_{1,q}
\end{array}\right]=\frac{1}{2\pi i}\int_{L}\frac{\prod\limits_{j=1}^m\Gamma(b_j-\beta_js)\prod\limits_{i=1}^l\Gamma(1-a_i+\alpha_is)}{\prod\limits_{i=l+1}^p\Gamma(a_i-\alpha_is)\prod\limits_{j=m+1}^q\Gamma(1-b_j+\beta_js)}z^{s}d s,
\end{equation}
for $z\in\mathbb{C}\setminus\{0\},$ where $m,l,p,q\in\mathbb{N}_0=\{0,1,2,\ldots\}$, $0\le m \le q$, $0\le l\le p$, $(m,l)\neq (0,0),$ $\alpha_i, \beta_j\in\mathbb{R}_+$, $a_i,b_j\in\mathbb{R}$ ($i=1,\ldots, p;j=1,\ldots,q$). Here $L$ is a suitable contour from $\gamma-i\infty$ to $\gamma+i\infty$, where $\gamma$ is a real number. The integral in (\ref{int1}) converges if the following conditions are met 
\begin{equation*}
\omega=\sum_{i=1}^{l}\alpha_i-\sum_{i=l+1}^{p}\alpha_i+\sum_{j=1}^{m}\beta_j-\sum_{j=m+1}^{q}\beta_j>0\mbox{ and } \lvert\arg z\rvert<\frac{\pi\omega}{2}. 
\end{equation*}
The Fox H-function vanishes for large $z$ because 
\begin{equation*}
\label{eq8}
H_{p,q}^{m,0}[z]\approx O\left(\exp\left(-\nu z^{\frac{1}{\nu}}\mu^{\frac{1}{\nu}}\right)z^{\frac{2\delta+1}{2\nu}}\right),
\end{equation*}
where $\mu=\prod\limits_{i=1}^p\alpha_i^{\alpha_i}\prod\limits_{j=1}^q\beta_j^{-\beta_j},$ $\delta=\sum\limits_{j=1}^q b_j-\sum\limits_{i=1}^p a_i+\frac{p-q}{2}$  and
$\nu=\sum\limits_{j=1}^{q}\beta_j-\sum\limits_{i=1}^{p}\alpha_i>0$. 
The generalized Wright function is defined (e.g. in \cite{kilb2005}) as
\begin{eqnarray}
\label{ub3}
{}_p\Psi_q\left[z\biggr\vert\begin{array}{c}
(a_i,\alpha_i)_{1,p}\\
(b_j,\beta_j)_{1,q}
\end{array}\right] & = & \sum_{k=0}^\infty\frac{\prod\limits_{i=1}^p\Gamma(a_i+\alpha_i k)}{\prod\limits_{j=1}^q\Gamma(b_j+\beta_j k)}\frac{z^k}{k!},
\end{eqnarray}
for $z\in\mathbb{C},$ $p,q\in\mathbb{N}_0,$ $a_i,b_j\in\mathbb{C}$ and $\alpha_i, \beta_j\in\mathbb{R}\setminus\{0\}$ ($i=1,\ldots,p; j=1,\ldots,q$). 
If $\Delta=\sum\limits_{j=1}^q \beta_j-\sum\limits_{i=1}^p \alpha_i>-1$ or $\Delta=-1$, then the series in (\ref{ub3}) is absolutely convergent for $z\in \mathbb{C}$ or $\lvert z\rvert<\prod\limits_{i=1}^p\lvert\alpha_i\rvert^{-\alpha_i}\prod\limits_{j=1}^q\lvert\beta_j\rvert^{\beta_j}$, respectively.

We know the following identities of H-function for $z>0$
\begin{eqnarray}
H^{m,l}_{p,q}\left[z\biggr\vert\begin{array}{c}
(A_i,\alpha_i)_{1,p}\\
(B_j,\beta_j)_{1,q}
\end{array}\right] & = & H^{l,m}_{q,p}\left[\frac{1}{z}\biggr\vert\begin{array}{c}
(1-B_j,\beta_j)_{1,q}\\
(1-A_i,\alpha_i)_{1,p}
\end{array}\right],\label{eq9}\\
H^{m,l}_{p,q}\left[z\biggr\vert\begin{array}{c}
(A_i,\alpha_i)_{1,p}\\
(B_j,\beta_j)_{1,q}
\end{array}\right] & = & k H^{m,l}_{p,q}\left[z^k\biggr\vert\begin{array}{c}
(A_i,k\alpha_i)_{1,p}\\
(B_j,k\beta_j)_{1,q}
\end{array}\right]\quad\mbox{ for } k>0,\label{eq10}\\
z^\sigma H^{m,l}_{p,q}\left[z\biggr\vert\begin{array}{c}
(A_i,\alpha_i)_{1,p}\\
(B_j,\beta_j)_{1,q}
\end{array}\right] & = & H^{m,l}_{p,q}\left[z\biggr\vert\begin{array}{c}
(A_i+\sigma\alpha_i,\alpha_i)_{1,p}\\
(B_j+\sigma\beta_j,\beta_j)_{1,q}
\end{array}\right]\quad\mbox{ for any } \sigma\in\mathbb{C},\label{eq11}\\
H^{m+r,l}_{p+1,q+r}\left[z\biggr\vert\begin{array}{c}
(A_i,\alpha_i)_{1,p},(1,r)\\
\left(\frac{j}{r},1\right)_{1,r},(B_j,\beta_j)_{1,q}
\end{array}\right] & = & \frac{(2\pi)^{\frac{r-1}{2}}}{\sqrt{r}}H^{m,l}_{p,q}\left[r^r z\biggr\vert\begin{array}{c}
(A_i,\alpha_i)_{1,p}\\
(B_j,\beta_j)_{1,q}
\end{array}\right]\quad\mbox{ for } r\in\mathbb{Z}_+, l\leq p .\label{eq12}
\end{eqnarray}
Moreover the following relations hold between the special functions. The Mittag-Leffler and Wright functions can be expressed by generalized Wright functions
\begin{equation}\label{eq13}
E_{\alpha,\beta}(z) = {}_1\Psi_1\left[z\left|\begin{array}{c}
(1,1)\\
(\beta,\alpha)
\end{array}\right.
\right] \quad\mbox{and}\quad
\Psi\left(z;\alpha,\beta\right) = \sum\limits_{k=1}^\infty \frac{1}{\Gamma(\alpha k+\beta)}\frac{z^k}{k!}={}_0\Psi_{1}\left[z\left|\begin{array}{c}
-\\
(\beta,\alpha)
\end{array}\right.\right].
\end{equation}

Let us state the main result of the study as the following theorem.
\begin{theorem}\label{main}
\begin{enumerate}
    \item For $0<\alpha<2$ and $d\neq 2$, the equation given in (\ref{appl}) has  the following solution
    \begin{equation}
    \label{nem1}
    u(x,t)=c_1x^aH_{1,m+2}^{m+2,0} \left[ \frac{x^{2-d}}{A(d-2)^2(\alpha+m)^{m}t^{\alpha+m}} \left| \begin{matrix}
( 1 , \alpha+m )  \\
\left( -\frac{s_1}{\alpha+m} , 1 \right) ,\left( -\frac{s_2}{\alpha+m} , 1 \right) , \left(\frac{j}{\alpha+m} , 1  \right)_{1,m}   \end{matrix}\right.\right];
    \end{equation}
    \item For $\alpha>2$ and $d\neq 2$, the equation given in (\ref{appl}) has the following solution
\begin{multline*}
u(x,t)= x^a\sum_{k=1}^{[\alpha]+1} c_kx^{\frac{(d-2)(\alpha-k)}{\alpha+m}}t^{\alpha-k}\\
   ~~~~~\times {}_{m+3}\Psi_1 \left[A(d-2)^2(\alpha+m)^{m}x^{d-2}t^{\alpha+m}\left|\begin{matrix}\left( \frac{\alpha-k-s_1}{\alpha+m} , 1 \right),\left( \frac{\alpha-k-s_2}{\alpha+m} , 1 \right) , \left(\frac{\alpha-k+i}{\alpha+m},1\right)_{1,m} , ( 1 , 1 )\\ 
( 1+\alpha-k , \alpha+m )  \end{matrix}\right.\right],
\end{multline*}

\item 
For $d= 2$, the equation given in (\ref{appl}) has the following solution
$$u(x,t)=x^a\sum_{k=1}^{[\alpha]+1} c_kt^{\alpha-k}{}_{m+1}\Psi_1 \left[Kt^{\alpha+m}\left|\begin{matrix} (\frac{\alpha-k+i}{\alpha+m},1)_{1,m} , ( 1 , 1 )\\ 
( \alpha-1 , \alpha+m )  \end{matrix}\right.\right],$$
\end{enumerate}
where $a$ is any real number, $K=Aa^2-Aa+Ba+C$ and
\begin{equation}
 s_{1,2}=\frac{\alpha+m}{2(2-d)}\left(\frac{B}{A}+2a-1\pm \sqrt{\left(1-\frac{B}{A}\right)^2-\frac{4C}{A}}\right).
\end{equation}
\end{theorem}

By setting $\alpha=1$ in (\ref{nem1}), we get
\begin{equation}
    \label{uuu1}
    u(x,t)=c_1x^aH_{1,m+2}^{m+2,0} \left[ \frac{x^{2-d}}{A(d-2)^2(1+m)^{m}t^{1+m}} \left| \begin{matrix}
( 1 , 1+m )  \\
\left( -\frac{s_1}{1+m} , 1 \right) ,\left( -\frac{s_2}{1+m} , 1 \right) , \left(\frac{j}{1+m} , 1  \right)_{1,m}   \end{matrix}\right.\right].
    \end{equation}
If we take $a=\frac{1}{2}\left(2d-\frac{B}{A}-3\pm\sqrt{\left(1-\frac{B}{A}\right)^2-\frac{4C}{A}}\right)$, (\ref{uuu1}) becomes 
\begin{multline*}
       u(x,t)=c_1x^{\frac{1}{2}\left(2d-\frac{B}{A}-3\pm\sqrt{\left(1-\frac{B}{A}\right)^2-\frac{4C}{A}
}\right)}\\
\times H_{1,m+2}^{m+2,0} \left[ \frac{x^{2-d}}{A(d-2)^2(1+m)^{m}t^{1+m}} \left| \begin{matrix}
( 1 , 1+m )  \\
\left(\frac{j}{1+m} , 1  \right)_{1,1+m},\left( \frac{d-2\pm\sqrt{\left(1-\frac{B}{A}\right)^2-\frac{4C}{A}
}}{d-2} , 1 \right)  \end{matrix}\right.\right].
    \end{multline*}
Then, we obtain the following corollary using (\ref{eq12}) and (1.125) of \cite{Hnom2} in above the solution.
\begin{corollary}
For $\alpha=1$ and $d\neq 2$, the equation given in (\ref{appl}) has the following solution
$$u(x,t)=cx^{-\frac{1}{2}\left(\frac{B}{A}-1\pm\sqrt{\left(1-\frac{B}{A}\right)^2-\frac{4C}{A}}\right)}t^{-\frac{(1+m)}{d-2}\left(d-2\pm\sqrt{ 
   \left(1-\frac{B}{A}\right)^2-\frac{4C}{A}           }\right)}EXP\left(-\frac{(1+m)x^{2-d}}{A(d-2)^2t^{1+m}}\right),$$
where $c$ is a constant.
\end{corollary}
The corollary shows that for a special case of (\ref{appl}), we have solutions expressed in terms of exponential functions.

Let us formulate  the  following  known  results as lemmas for  the  generalized  Wright  functions \cite{kilb2005,Kiryakova, Khon} and Fox H-functions \cite{Hnom2, glock, Khon} for the convenience to prove our main result.
\begin{lemma}\label{le3} 
Let $\Delta=\sum\limits_{j=1}^q B_j-\sum\limits_{i=1}^p A_i>-1.$ Then the following equalities hold for $\alpha \in\mathbb{R}_+$ and $a\in\mathbb{R}.$

(1) If $\beta_1>0$, $B_1>0$ and $A_1=\alpha_1=1,$ then  we have
\begin{eqnarray*}
 & & \frac{d^\alpha}{d z^\alpha}\left(z^{B_1-1}{}_p\Psi_q\left[a z^{\beta_1}\biggr\vert\begin{array}{c}
(1,1), (A_i,\alpha_i)_{2,p}\\
(B_j,\beta_j)_{1,q}
\end{array}\right]\right)\\
& = & a^m z^{B_1+m\beta_1-1-\alpha}{}_{p}\Psi_{q}\left[az^{\beta_1}\biggr\vert\begin{array}{c}
(1,1), (A_i+m\alpha_i,\alpha_i)_{2,p}\\
(B_1+m\beta_1-\alpha,\beta_1), (B_j+m\beta_j,\beta_j)_{2,q}
\end{array}\right],
\end{eqnarray*}
where $m$ is the smallest non-negative integer such that $B_1+m\beta_1 -\alpha-1$ is not a negative integer.

(2) For $\sigma\in\mathbb{R}\setminus\{0\}$ and $R\in\mathbb{R},$ the following equality holds
\begin{multline*}
\left(\frac{1}{\alpha}z\frac{d}{d z}+R\right)\left(z^{\frac{A_1\sigma}{\alpha_1}-\alpha R}{}_p\Psi_q\left[az^{\sigma}\biggr\vert\begin{array}{c}
(A_i,\alpha_i)_{1,p}\\
(B_j,\beta_j)_{1,q}
\end{array}\right]\right)
\\=
\frac{\sigma}{\alpha_1\alpha} z^{\frac{A_1\sigma}{\alpha_1}-\alpha R}{}_p\Psi_q\left[az^{\sigma}\biggr\vert\begin{array}{c}
(A_1+1,\alpha_1), (A_i,\alpha_i)_{2,p}\\
(B_j,\beta_j)_{1,q}
\end{array}\right].
\end{multline*}
\end{lemma}
\begin{lemma}\label{le5}
Let $\nu=\sum_{j=1}^{q}\beta_j-\sum_{i=1}^{p}\alpha_i> 0,$ $\mu=\sum\limits_{j=1}^{m}\beta_j-\sum\limits_{j=m+1}^{q}\beta_j-\sum\limits_{i=1}^{p}\alpha_i>0$, $\alpha\in \mathbb{R}_+$ and $a\in \mathbb{R}\setminus\{0\}$.  Then the following equalities hold.

(1) If $a>0$, then
\begin{multline*}
\frac{d^\alpha}{d z^\alpha}H_{p,q}^{m,0}\left[az^{-\alpha_p}\biggr\vert\begin{array}{c}
(A_i,\alpha_i)_{1,p-1}, (1, \alpha_p)\\
(B_j,\beta_j)_{1,q}
\end{array}\right]\\
=z^{-\alpha}H_{p,q}^{m,0}\left[az^{-\alpha_p}\biggr\vert\begin{array}{c}
(A_i,\alpha_i)_{1,p-1}, (1-\alpha, \alpha_p)\\
(B_j,\beta_j)_{1,q}
\end{array}\right],\quad z>0.
\end{multline*}

(2) If $m\ge 1,$ then
\begin{multline*}
\left(\frac{\beta_1}{\alpha_p}z\frac{d}{dz}+B_1\right)H_{p,q}^{m,0}\left[az^{-\alpha_p}\biggr\vert\begin{array}{c}
(A_i,\alpha_i)_{1,p-1}, (1, \alpha_p)\\
(B_j,\beta_j)_{1,q}
\end{array}\right]\\
=H_{p,q}^{m,0}\left[az^{-\alpha_p}\biggr\vert\begin{array}{c}
(A_i,\alpha_i)_{1,p-1}, (1, \alpha_p)\\
(B_1+1,\beta_1),(B_j,\beta_j)_{2,q}
\end{array}\right].
\end{multline*}
\end{lemma}

\section{Proof of Main Theorem}
First, we will derive solutions to the following equations
\begin{equation}\label{a1}
\frac{d^\alpha}{dz^\alpha}y(z)=
  z^m\left(a_nz^n\frac{d^n}{dz^n}y+a_{n-1}z^{n-1}\frac{d^{n-1}}{dz^{n-1}}y+\cdots+ a_1z\frac{d}{dz}y +a_0y\right), \quad z>0,
\end{equation} 
where $\alpha$ and $a_n$ ($i=0,\ldots,n$) are positive real numbers, $a_1,a_2,\ldots,a_{n-1}$ are real numbers, $m\in\mathbb{N}$. Here the fractional differentiation is defined in Riemann-Liouville manner. In \cite{Khon}, we derived exact solutions to the class of linear fractional differential equations of (\ref{a1}) when $m=0$. 

 In order to obtain the solutions of the equation given in (\ref{a1}), we consider the following characteristic polynomial
\begin{equation}\label{char}
    \widetilde  P(s)=a_n\prod\limits_{j=0}^{n-1}(s-j)+ a_{n-1}\prod\limits_{j=0}^{n-2}(s-j)+ \cdots+ a_1s +a_0
\end{equation}
  associated with the differential operator
  $$
  P(y)=a_nz^n\frac{d^n}{dz^n}y+a_{n-1}z^{n-1}\frac{d^{n-1}}{dz^{n-1}}y+\cdots+ a_1z\frac{d}{dz}y +a_0y.
  $$
  Let $s_1,\dots,s_n$ be the roots of the characteristic polynomial (\ref{char}). Then we can rewrite the right hand side of (\ref{a1}) in factorized form as following
  \begin{multline*}
    a_nz^n\frac{d^n}{dz^n}y+a_{n-1}z^{n-1}\frac{d^{n-1}}{dz^{n-1}}y+\cdots+ a_1z\frac{d}{dz}y +a_0y\\
    =a_n\left(z\frac{d}{dz}-s_n\right)\cdots\left(z\frac{d}{dz}-s_2\right)\left(z\frac{d}{dz}-s_1\right)y.
\end{multline*}
To present the solutions of the equation (\ref{a1}), as a preparation we show that the following identity holds.
\begin{lemma}\label{lem1}
Let $a\in\mathbb{R}_{+}, m\in \mathbb{N}$ and $b\in\mathbb{C}$. Then
\begin{equation}\label{lemeq}
    \frac{1}{\Gamma(1+ab+m)}\prod_{i=1}^{m}\Gamma\left(\frac{i}{a}+b+1\right)=\frac{1}{a^m\Gamma(1+ab)}\prod_{i=1}^m\Gamma\left(\frac{i}{a}+b\right).
\end{equation}
\end{lemma}
\begin{proof}
By noticing the following identities
$$
\prod_{i=1}^m\left(\frac{i}{a}+b\right)\Gamma\left(\frac{i}{a}+b\right)=\frac{1}{a^m}\prod_{i=1}^m(i+ab)\prod_{i=1}^m\Gamma\left(\frac{i}{a}+b\right)=\frac{\Gamma(1+ab+m)}{a^m\Gamma(1+ab)}\prod_{i=1}^m\Gamma\left(\frac{i}{a}+b\right).$$
Hence, the proof is completed. 
\end{proof}
Now we are ready to give the solutions of (\ref{a1}).
\begin{proposition}\label{thm1}
The equation given in (\ref{a1}) has the following solutions.
\begin{enumerate}
    \item For $\alpha<n,$ the solution is expressed as
    \begin{equation}\label{uuu2}
    y(z)=c_1H_{1,m+n}^{m+n,0} \left[ \frac{z^{-(\alpha+m)}}{a_n(\alpha+m)^{m+n}} \left| \begin{matrix}
( 1 , \alpha+m )  \\
\left( -\frac{s_j}{\alpha+m} , 1 \right)_{1,n}, \left(\frac{j}{\alpha+m} , 1  \right)_{1,m}   \end{matrix}\right.\right];  
    \end{equation}
    \item For $\alpha>n,$ the solution is expressed as
    \begin{equation}\label{uuu3}
        y(z)=\sum_{k=1}^{[\alpha]+1} c_kz^{\alpha-k} {}_{m+n+1}\Psi_1 \left[a_{n}(\alpha+m)^{m+n}z^{\alpha+m}\left|\begin{matrix}\left( \frac{\alpha-k-s_i}{\alpha+m} , 1 \right)_{1,n}, \left(\frac{\alpha-k+i}{\alpha+m},1\right)_{1,m} , ( 1 , 1 )\\ 
( 1+\alpha-k , \alpha+m )  \end{matrix}\right.\right].
    \end{equation}
Here  $s_1,\dots,s_n$ are the roots of the characteristic polynomial (\ref{char}). 
\end{enumerate}
\end{proposition}
\begin{proof}
Because the convergence
conditions of H-functions and generalized wright functions holds the functions in (\ref{uuu2}) and (\ref{uuu3}), it is sufficient to show that the functions satisfy the equation in (\ref{a1}).
Firstly, we rewrite the right hand side of the equation (\ref{a1}) in the factorized form
\begin{align}\label{uuu5}
    & z^{m}\left(a_nz^n\frac{d^n}{dz^n}+a_{n-1}z^{n-1}\frac{d^{n-1}}{dz^{n-1}}+\cdots+ a_1z\frac{d}{dz} +a_0\right)\nonumber\\
    &~~~~\left( c_1H_{1,m+n}^{m+n,0} \left[ \frac{z^{-(\alpha+m)}}{a_n(\alpha+m)^{m+n}} \left| \begin{matrix}
( 1 , \alpha+m )  \\
 \left( -\frac{s_j}{\alpha+m} , 1 \right)_{1,n}, & \left(\frac{j}{\alpha+m} , 1  \right)_{1,m} &  \end{matrix}\right.\right]\right)\nonumber\\ 
    &~=c_1a_n(\alpha+m)^n z^m\left(\frac{1}{\alpha+m}z\frac{d}{dz}-\frac{s_n}{\alpha+m}\right)\cdots\left(\frac{1}{\alpha+m}z\frac{d}{dz}-\frac{s_1}{\alpha+m}\right)\nonumber\\
    & ~~~~\left( H_{1,m+n}^{m+n,0} \left[ \frac{z^{-(\alpha+m)}}{a_n(\alpha+m)^{m+n}} \left| \begin{matrix}
( 1 , \alpha+m )  \\
 \left( -\frac{s_j}{\alpha+m} , 1 \right)_{1,n}, & \left(\frac{j}{\alpha+m} , 1  \right)_{1,m} &  \end{matrix}\right.\right]\right).
\end{align}
From Lemma~\ref{le5}, we get the following identity for $k=1,\ldots, n$
\begin{align*}
& \left(\frac{1}{\alpha+m}z\frac{d}{dz}-\frac{s_k}{\alpha+m}\right)\left(H_{1,m+n}^{m+n,0} \left[ \frac{z^{-(\alpha+m)}}{a_n(\alpha+m)^{m+n}} \left| \begin{matrix}
( 1 , \alpha+m )  \\
\left( -\frac{s_j}{\alpha+m} , 1 \right)_{1,n}& \left(\frac{j}{\alpha+m} , 1  \right)_{1,m} &  \end{matrix}\right.\right]\right)\\
&=H_{1,m+n}^{m+n,0} \left[ \frac{z^{-(\alpha+m)}}{a_n(\alpha+m)^{m+n}} \left| \begin{matrix}
( 1 , \alpha+m )  \\
\left( -\frac{s_j}{\alpha+m} , 1 \right)_{1,k-1}, \left(1 -\frac{s_k}{\alpha+m} , 1 \right), \left( -\frac{s_j}{\alpha+m} , 1 \right)_{k+1,n}, & \left(\frac{j}{\alpha+m} , 1  \right)_{1,m} &  \end{matrix}\right.\right].
\end{align*}
Then applying the above identity in (\ref{uuu5}), the right hand side of the equation becomes
\begin{align*}
& c_1a_nz^m(\alpha+m)^n H_{1,m+n}^{m+n,0} \left[ \frac{z^{-(\alpha+m)}}{a_n(\alpha+m)^{m+n}} \left| \begin{matrix}
( 1 , \alpha+m )  \\
 ( 1-\frac{s_j}{\alpha+m} , 1 )_{1,n} & (\frac{j}{\alpha+m} , 1  )_{1,m} &  \end{matrix}\right.\right]. 
\end{align*}
Now, let us begin the calculation of left hand side of the equation. By applying the Lemma~\ref{le5}, we get
\begin{equation*}
 \frac{d^{\alpha}}{dz^{\alpha}}y=c_1z^{-\alpha}H_{1,m+n}^{m+n,0} \left[ \frac{z^{-(\alpha+m)}}{a_n(\alpha+m)^{m+n}} \left| \begin{matrix}
( 1-\alpha , \alpha+m )\\( -\frac{s_j}{\alpha+m} , 1 )_{1,n}, (\frac{j}{\alpha+m} , 1  )_{1,m}  \end{matrix}\right.\right].   
\end{equation*}
Then applying (\ref{eq11}),  the left hand side is
$$c_1a_nz^m(\alpha+n)^{m+n}H_{1,m+n}^{m+n,0} \left[ \frac{z^{-(\alpha+m)}}{a_n(\alpha+m)^{m+n}} \left| \begin{matrix}
( 1+m , \alpha+m )\\( 1-\frac{s_j}{\alpha+m} , 1 )_{1,n}, (1+\frac{j}{\alpha+m} , 1  )_{1,m}  \end{matrix}\right.\right].$$
Hence, the left hand side becomes
\begin{equation*}
\frac{d^{\alpha}}{dz^{\alpha}}y=
c_1a_nz^m(\alpha+n)^{m+n}\frac{1}{2\pi i}\int_{L}{\frac{\prod\limits_{j=1}^n\Gamma\left( 1-\frac{s_j}{\alpha+m}-s\right)\prod\limits_{j=1}^{m} \Gamma\left(1+\frac{j}{\alpha+m}-s\right)}{\Gamma\left(1+m-(\alpha+m)s\right)}\cdot\left(\frac{z^{-(\alpha+m)}}{a_n(\alpha+m)^{m+n}}\right)^{s}\,d} s
\end{equation*}
By substituting $a=\alpha+m,$ $b=-s$ in Lemma~\ref{lem1}, we see that
\begin{align*}
&\frac{\prod\limits_{j=1}^{m} \Gamma\left(1+\frac{j}{\alpha+m}-s\right)}{\Gamma(1+m-(\alpha+m)s)}= \frac{1}{(\alpha+m)^m}\cdot\frac{\prod\limits_{j=1}^{m}\Gamma\left(\frac{j}{\alpha+m}-s\right)}{\Gamma(1-(\alpha+m)s)},
\end{align*}
which concludes the right and left hand sides of the equations are equal.

The second assertion of the theorem consists of $[\alpha]+1$ summands. Since the equation is linear, it is sufficient to show that it holds for single $k$th summand. For the fixed summand, we  will take the similar steps as in the previous case. From Lemma~\ref{le5}, we get the following identity for $l=1,\ldots, n$
\begin{align*}
& \left(\frac{1}{\alpha}z\frac{d}{dz}-\frac{s_l}{\alpha}\right)\left(   z^{\alpha-k}{}_{m+n+1}\Psi_1\left[a_{n}(\alpha+m)^{m+n}z^{\alpha+m}\left|\begin{matrix} ( \frac{\alpha-k-s_i}{\alpha+m} , 1 )_{1,n}, (\frac{\alpha-k+i}{\alpha+m},1)_{1,m} , ( 1 , 1 )\\ 
( 1+\alpha-k , \alpha+m )  \end{matrix}\right.\right] \right)=\frac{\alpha+m}{\alpha}z^{\alpha-k}\\
&~ \times {}_{m+n+1}\Psi_1 \left[a_{n}(\alpha+m)^{m+n}z^{\alpha+m}\left|\begin{matrix}( \frac{\alpha-k-s_i}{\alpha+m} , 1 )_{1,l-1}, ( \frac{\alpha-k-s_l}{\alpha+m}+1 , 1 ), ( \frac{\alpha-k-s_i}{\alpha+m} , 1 )_{l+1,n}, (\frac{\alpha-k+i}{\alpha+m},1)_{1,m} , ( 1 , 1 )\\ 
( 1+\alpha-k , \alpha+m )  \end{matrix}\right.\right],
   \end{align*}
the right hand side of the equation becomes
\begin{align*}
&a_n \alpha^{n}z^m\left(\frac{1}{\alpha}z\frac{d}{dz}-\frac{s_n}{\alpha}\right) \cdots\left(\frac{1}{\alpha}z\frac{d}{dz}-\frac{s_2}{\alpha}\right)\left(\frac{1}{\alpha}z\frac{d}{dz}-\frac{s_1}{\alpha}\right)\\
&~~~~\left(   z^{\alpha-k}{}_{m+n+1}\Psi_1 \left[a_{n}(\alpha+m)^{m+n}z^{\alpha+m}\left|\begin{matrix}( \frac{\alpha-k-s_i}{\alpha+m} , 1 )_{1,n}, (\frac{\alpha-k+i}{\alpha+m},1)_{1,m} , ( 1 , 1 )\\ 
( 1+\alpha-k , \alpha+m )  \end{matrix}\right.\right] \right)\\
& =a_n (\alpha+m)^{n} z^{\alpha+m-k}{}_{m+n+1}\Psi_1 \left[a_{n}(\alpha+m)^{m+n}z^{\alpha+m}\left|\begin{matrix}( \frac{\alpha-k-s_i}{\alpha+m}+1 , 1 )_{1,n}, (\frac{\alpha-k+i}{\alpha+m},1)_{1,m} , ( 1 , 1 )\\ 
( 1+\alpha-k , \alpha+m )  \end{matrix}\right.\right].
   \end{align*}
   The left hand side of the equation is calculated as following
    \begin{align*}
   &\frac{d^{\alpha}}{dz
^{\alpha}}z^{\alpha-k} {}_{m+n+1}\Psi_1
\left[a_{n}(\alpha+m)^{m+n}z^{\alpha+m}\left|\begin{matrix}(
\frac{\alpha-k-s_i}{\alpha+m} ,1)_{1,n},(\frac{\alpha-k+i}{\alpha+m},1)_{1,m},( 1 , 1 )\\ 
( 1+\alpha-k , \alpha+m ) \end{matrix}\right.\right]\\
&~ =b_0a_n(\alpha+m)^{m+n}z^{\alpha+m-k}\\
&~~~~\times{}_{m+n+2}\Psi_2
\left[a_{n}(\alpha+m)^{m+n}z^{\alpha+m}\left|\begin{matrix}( 1 , 1) , (
\frac{\alpha-k-s_i}{\alpha+m}+1 ,1)_{1,n},(\frac{\alpha-k+i}{\alpha+m}+1,1)_{1,m},  ( 2 , 1 )\\ 
( 2 , 1),( 1+\alpha-k+m , \alpha+m ) \end{matrix}\right.\right]\\
&~ =b_0a_n(\alpha+m)^{m+n}z^{\alpha+m-k}\\
&~~~~\times {}_{m+n+1}\Psi_1
\left[a_{n}(\alpha+m)^{m+n}z^{\alpha+m}\left|\begin{matrix} (
\frac{\alpha-k-s_i}{\alpha+m}+1 ,1)_{1,n},(\frac{\alpha-k+i}{\alpha+m}+1,1)_{1,m},  ( 1 , 1 )\\ ( 1+\alpha-k+m , \alpha+m ) \end{matrix}\right.\right]\\
&~ =b_0a_n(\alpha+m)^{m+n}z^{\alpha+m-k}\\
&~~~~\times\sum \limits_{j=0}^{\infty}\frac{ \prod\limits_{i=1}^{n}\Gamma
\left( \frac{\alpha-k-s_i}{\alpha+m}+1+j\right)  \prod\limits_{i=1}^{m}\Gamma  \left(
\frac{\alpha-k+i}{\alpha+m}+1+j\right) \Gamma(1+j)  }{\Gamma(1+\alpha-k+m+(\alpha+m)j)}\cdot\frac{
\left(a_{n}(\alpha+m)^{m+n}z^{\alpha+m}\right)^{j}}{j!}
\end{align*}
which equals to the right hand side of the equation via Lemma~\ref{lem1} with $a=\alpha+m,~b=\frac{\alpha-k}{\alpha+m}+j$.
\end{proof}
So, we have formulated exact solutions to the fractional differential equations given in (\ref{a1}). In the following, we will demonstrate how to apply the results to (\ref{appl}).

By applying the substitution
\begin{equation}\label{ansatz}
u=x^{a}\varphi(x^{\frac{d-2}{\alpha+m}} t),\quad  x^{\frac{d-2}{\alpha+m}}t=z, ~(a\in\mathbb{R})
\end{equation}
in (\ref{appl}), we have
\begin{multline}\label{a3}
    \frac{d^{\alpha}}
  {dz^{\alpha}}\varphi=
 z^m \left[\frac{A(d-2)^2}{(\alpha+m)^2}z^2\varphi_{zz}+\frac{d-2}{\alpha+m}\left(\frac{A(d-2)}{\alpha+m}+B\right.\right.\\ \Big.\Big.+2Aa-A\Big)z\varphi_{z}+(Aa^2-Aa+Ba+C)\varphi\Big].
\end{multline}
Let's take a look at the following characteristic equation of (\ref{a3})
\begin{equation*}
  \frac{A(d-2)^2}{(\alpha+m)^2}s(s-1) +\frac{d-2}
  {\alpha+m}\left(\frac{A(d-2)}{\alpha+m}+B+A(2a-1)\right)s+Aa(a-1)+Ba+AC=0   
    \end{equation*}
where $d\neq 2$. Finally, according to Proposition~\ref{thm1} we obtain the following solutions to equation (\ref{a3})
\begin{enumerate}
    \item For $0<\alpha<2,$ 
    \begin{equation}
    \varphi(z)=c_1H_{1,m+2}^{m+2,0} \left[ \frac{z^{-(\alpha+m)}}{A(d-2)^2(\alpha+m)^{m}} \left| \begin{matrix}
( 1 , \alpha+m )  \\
( -\frac{s_1}{\alpha+m} , 1 ) ,( -\frac{s_2}{\alpha+m} , 1 ), (\frac{j}{\alpha+m} , 1  )_{1,m}   \end{matrix}\right.\right],    
    \end{equation}
    \item For $\alpha>2,$
    \begin{equation}
   \varphi(z)= \sum_{k=1}^{[\alpha]+1} c_kz^{\alpha-k} {}_{m+3}\Psi_1 \left[A(d-2)^2(\alpha+m)^{m}z^{\alpha+m}\left|\begin{matrix}( \frac{\alpha-k-s_1}{\alpha+m} , 1 ),( \frac{\alpha-k-s_2}{\alpha+m} , 1 ) , (\frac{\alpha-k+i}{\alpha+m},1)_{1,m} , ( 1 , 1 )\\ 
( 1+\alpha-k , \alpha+m )  \end{matrix}\right.\right].      
    \end{equation}
\end{enumerate}
The case $d=2$, we get the following solutions to equation (\ref{a3})
 \begin{equation}
   \varphi(z)= \sum_{k=1}^{[\alpha]+1} c_kz^{\alpha-k} {}_{m+1}\Psi_1 \left[(Aa^2-Aa+Ba+C)(\alpha+m)^{m}z^{\alpha+m}\left|\begin{matrix} (\frac{\alpha-k+i}{\alpha+m},1)_{1,m} , ( 1 , 1 )\\ 
( 1+\alpha-k , \alpha+m )  \end{matrix}\right.\right].      
    \end{equation}
Under the substitution (\ref{ansatz}) we get our desired results in Theorem~\ref{main}.

\section*{Acknowledgments}

This work was partially supported by the National University of Mongolia (Grant No.P2018-3584) and by the Mongolian
Foundation for Science and Technology (Grant
No.SHUTBIKHKHZG-2022/164).

\end{document}